\documentclass[12pt,oneside]{amsart}

\usepackage[total={5.5in, 7.5in}]{geometry}   
\usepackage{todonotes}
\usepackage{color}
\usepackage{mathrsfs}
\usepackage[mathscr]{euscript}
\usepackage{bm}
\usepackage{bbm}
\usepackage{wasysym}
\usepackage[utf8]{inputenc}
\inputencoding{utf8}
\usepackage{comment}
\usepackage{graphicx}
\usepackage[justification=centering]{caption}
\usepackage{amsthm,amssymb,amsmath}
\usepackage{subfigure}
\allowdisplaybreaks[4]
\usepackage{comment}

\newtheorem{conjecture}{Conjecture}

\newtheorem{theorem}{Theorem}[section]
\newtheorem{lemma}[theorem]{Lemma}

\newtheorem{proposition}[theorem]{Proposition}
\newtheorem{corollary}[theorem]{Corollary}

\theoremstyle{definition}
\newtheorem{definition}[theorem]{Definition}

\def\NN{\mathbb{N}}
\def\CC{\mathbb{C}}

\def\d{\,\mathrm{d}}
\def\N{\mathcal{N}}
\def\MM{\mathcal{M}}

\def\RR{\mathbb{R}}
\def\TT{\mathbb{T}}
\def\ZZ{\mathbb{Z}}

\def\M{\mathscr{M}}
\def\tri{\,\triangle\,}

\newcommand{\al}{\alpha}
\newcommand{\be}{\beta}

\newcommand{\Ga}{\Gamma}

\newcommand{\De}{\Delta}

\newcommand{\La}{\Lambda}

\newcommand{\om}{\omega}
\newcommand{\Om}{\Omega}
\newcommand{\eps}{\epsilon}

\newcommand{\cp}{\mathcal P}
\newcommand{\ch}{\mathcal H}

\newcommand{\cn}{\mathcal N}

\newcommand{\wt}{\widetilde}
\newcommand{\wh}{\widehat}

\newcommand{\ZC}{\mathbb{C}}

\newcommand{\ZN}{\mathbb{N}}
\newcommand{\ZS}{\mathbb{S}}
\newcommand{\ti}{\tilde}

\title{A note on the largest sum-free sets of integers}

\author{Yifan Jing}
\address{Mathematical Institute, University of Oxford, UK}
\email{yifan.jing@maths.ox.ac.uk}
\author{Shukun Wu}
\address{Department of Mathematics, Department of Mathematics, California Institute of Technology, USA}
\email{skwu@caltech.edu}

\thanks{YJ was supported by Ben Green’s Simons Investigator Grant, ID:376201, by the Arnold O. Beckman Research Award (UIUC Campus Research Board RB21011), by the University Fellowship from UIUC, and by the Trijitzinsky Fellowship.}
\thanks{SW was supported by the Arnold O. Beckman Research Award (UIUC Campus Research Board RB21011).}

\subjclass[2010]{Primary 11B30; Secondary  11K70}
\date{}
\begin{document}

\begin{abstract}
Given $A$ a set of $N$ positive integers, an old question in additive combinatorics asks that whether $A$ contains a sum-free subset of size at least $N/3+\omega(N)$ for some increasing unbounded function $\omega$. The question is generally attacked in the literature by considering another conjecture, which asserts that as $N\to\infty$, $\max_{x\in\RR/\ZZ}\sum_{n\in A}(\mathbbm{1}_{(1/3,2/3)}-1/3)(nx)\to\infty$. 
This conjecture, if true, would also imply that a similar phenomenon occurs for $(2k,4k)$-sum-free sets for every $k\geq1$. In this note, we prove the latter result directly. The new ingredient of our proof is a structural analysis on the host set $A$, which might be of independent interest.
\end{abstract}
\maketitle

\section{Introduction}

Given $X$ a set of positive integers, we say $X$ is \emph{sum-free} if there does not exist $x_1,x_2,y_1$ in $X$ with $x_1+x_2= y_1$. The study of sum-free sets can be tracked back to Schur~\cite{Schur}, where he used extremal properties of sum-free sets 
to prove that the Fermat's last theorem does not hold in $\mathbb{F}_p$. 

Let $A$ be a finite subset of an abelian group $G$. A natural question is:  How large must the maximal sum-free subset of $A$ be? We use $\mathscr{M}_{(2,1)}(A)$ in this note to denote the size of maximal sum-free subset of $A$. When $A=G$, interest in determining $\mathscr{M}_{(2,1)}(G)$ for finite abelian groups $G$ goes back over 50 years. In 1969, Diananda and Yap~\cite{DY69} determined the size of the maximal sum-free set in $G$ whenever $|G|$ has a prime factor $p\not\equiv 1\pmod 3$, and the question is completely solved for all finite abelian groups recently by Green and Ruzsa~\cite{BR05}.
There is a large body of literature on extremal problems of sum-free sets. 
For example, see \cite{ABMS, BLST, G04, LS21} for counting sum-free sets, and \cite{JL20, HJ18, Hongliu} for the Erd\H os--Rothschild problems of sum-free sets. 

In this note, we are interested in finding the largest sum-free subset of $A$ when $A$ is a set of $N$ integers. More precisely, we define
\[
\mathscr{M}_{(2,1)}(N)=\inf_{A\subseteq \ZZ^{>0},|A|=N}\mathscr{M}_{(2,1)}(A).
\]

The study of $\mathscr{M}_{(2,1)}(N)$ originates with
Erd\H os~\cite{Erdos65}, who showed that $\mathscr{M}_{(2,1)}(N)\geq N/3$. An upper bound of the
same asymptotic quality was achieved in a recent breakthrough in the
upper bound by Eberhard, Green, and Manners~\cite{EGM}, who showed that
$\mathscr{M}_{(2,1)}(N)\leq N/3+o(N)$. Achieving any substantial improvement to
Erd\H os' lower bound is a long-standing open problem. In particular
the following conjecture is made in a series of papers \cite{Bourgain97,EGM,Erdos65}; for more background we refer to the recent survey by Tao and Vu~\cite{TV17}. 

\begin{conjecture}[The sum-free conjecture, combinatorial form]\label{conj:main}
There is a function $\omega(N)\to\infty$ as $N\to\infty$, such that
\[
\mathscr{M}_{(2,1)}(N)\geq\frac{N}{3}+\omega(N). 
\]
\end{conjecture}

In \cite{Erdos65}, using a probabilistic argument, Erd\H os showed that $\mathscr{M}_{(2,1)}(N)\geq N/3$. This argument is actually not complicated: Let $\Omega$ be a sum-free subset of $\RR/\ZZ$, for example, the interval $(1/3,2/3)$, then $(\Omega+\Omega)\cap\Omega=\varnothing$. \footnote{In fact by Kneser's inequality~\cite{Kneser}, we have an upper bound on $|\Omega|$ that $|\Omega|\leq1/3$.} 
For any $x\in\RR/\ZZ$, we let $A_x$  be the set of integers $n$ in $A$ such that $nx\in\Omega$. Then clearly $A_x$ is sum-free, and hence we have
\[
\mathscr{M}_{(2,1)}(A)\geq |A_x|= \sum_{n\in A}\mathbbm{1}_\Omega(nx),
\]
where $\mathbbm{1}_\Omega$ is the characteristic function of $\Omega$. 
When $x$ is chosen randomly from $\RR/\ZZ$, the expected size of $A_x$ is $N/3$, which implies that $|A_x|\geq N/3$ for some $x$. 

The lower bound estimate of $\max_{x\in\RR/\ZZ}\sum_{n\in A}\mathbbm{1}_{\Omega}(nx)$ was later improved to $(N+1)/3$ by Alon and Kleitman~\cite{AK90}, and the best estimate to date was obtained by Bourgain~\cite{Bourgain97}, who showed that $\max_{x\in\RR/\ZZ}\sum_{n\in A}\mathbbm{1}_{\Omega}(nx)\geq (N+2)/3$. Recently Shakan \cite{Shakan22} obtained a different proof for the bound $(N+2)/3$. 
In particular, the following conjecture would imply Conjecture~\ref{conj:main}.

\begin{conjecture}[The sum-free conjecture, analytic form]\label{conj:main2}
There is a function $\omega(N)\to\infty$ as $N\to\infty$, such that for every set $A$ of $N$ positive integers, there exists a maximal sum-free set $\Omega\subseteq\RR/\ZZ$ that
\[
\max_{x\in\RR/\ZZ}\sum_{n\in A}\Big(\mathbbm{1}_{\Omega}-\frac{1}{3}\Big)(nx)=\omega(N).
\]
\end{conjecture}

The analogous conjectures for $(k,\ell)$-sum-free sets are also well-studied. By \emph{$(k,\ell)$-sum-free} we mean that for every $k+\ell$ elements $x_1,\dots,x_k,y_1,\dots,y_\ell$, we always have $\sum_{i=1}^k x_i\neq\sum_{j=1}^\ell y_j$ (throughout the paper we always assume that $k<\ell$). We can similar define $\mathscr{M}_{(k,\ell)}(A)$ to be the size of a maximal $(k,\ell)$-sum-free subset of $A$, and let
\[
\mathscr{M}_{(k,\ell)}(N)=\inf_{A\subseteq \ZZ^{>0},|A|=N}\mathscr{M}_{(k,\ell)}(A).
\]
Recently, Eberhard~\cite{Eberhard15} showed that $\mathscr{M}_{(1,k)}(N)=N/(1+k)+o(N)$. The result is later generalized by the authors~\cite{JW20} for all $(k,\ell)$, that $\mathscr{M}_{(k,\ell)}(N)=N/(k+\ell)+o(N)$. It remains to show whether $\mathscr{M}_{(k,\ell)}(N)\geq N/(k+\ell)+\omega(N)$ for every pair $(k,\ell)$, where $\omega(N)$ is a function that tends to infinity as $N$ tends to infinity. Using the probabilistic argument by Erd\H os, one can immediately get $\mathscr{M}_{(k,\ell)}(N)\geq N/(k+\ell)$. In general, we believe the following should be true, which is a generalization of Conjecture~\ref{conj:main2} to all $(k,\ell)$-sum-free sets.

\begin{conjecture}\label{conj:main3}
There is a function $\omega(N)\to\infty$ as $N\to\infty$, such that for every set $A$ of $N$ positive integers, there exists a maximal $(k,\ell)$-sum-free set $\Omega(k,\ell)\subseteq\RR/\ZZ$, and we have
\[
\max_{x\in\RR/\ZZ}\sum_{n\in A}\Big(\mathbbm{1}_{\Omega(k,\ell)}-\frac{1}{k+\ell}\Big)(nx)=\omega(N).
\]
\end{conjecture}

Conjecture~\ref{conj:main3} is verified by Bourgain~\cite{Bourgain97} for $(1,3)$-sum-free sets, and recently by the authors~\cite{JW20} for $(k,5k)$-sum-free sets for every $k\geq1$. In particular, the authors observed that if Conjecture~\ref{conj:main3} holds for $(k,\ell)$-sum-free sets, then it also holds for $(km,\ell m)$-sum-free sets for every $m\geq2$. Thus, if Conjecture~\ref{conj:main2} holds, then this would imply that Conjecture~\ref{conj:main3} holds for $(2k,4k)$-sum-free sets for all $k\geq1$. Hence the authors believe that the $(2k,4k)$-sum-free problem is one of the most interesting cases of Conjecture~\ref{conj:main3}. 
 In this paper, we prove the $(2k,4k)$-sum-free case without assuming Conjecture~\ref{conj:main2}. 

\begin{theorem}\label{thm:main}
For every $k\geq1$, there is a function $\omega(N)=\log N/\log\log N$, such that for every set $A$ of $N$ positive integers, there exists a maximal $(2k,4k)$-sum-free set $\Omega(2k,4k)\subseteq\RR/\ZZ$, and we have
\[
\max_{x\in\RR/\ZZ}\sum_{n\in A}\Big(\mathbbm{1}_{\Omega(2k,4k)}-\frac{1}{6k}\Big)(nx)\gg \omega(N).
\]
As a consequence, there is an absolute  constant $c>0$, such that 
\[
\mathscr{M}_{(2k,4k)}(N)\geq \frac{N}{6k}+c\,\omega(N).
\]
\end{theorem}

The new ingredients used in proving Theorem~\ref{thm:main} contain a structural analysis of the given set $A$. 
Recall that 
a F\o lner sequence in $(\NN,\cdot)$ is a collection of sets of integers $\{F_n\}_{n=1}^\infty$, such that for every $a\in\NN^{>0}$,
\[
\lim_{n\to\infty}\frac{|F_n\tri (a\cdot F_n)|}{|F_n|}=0.
\]
Thus, when $A$ is close to a set in a F\o lner sequence, we expect that $|A\tri (a\cdot A)|$ is small for appropriate $a$. Inspired by the structure of F\o lner sequences (which is the only known constructive example whose largest $(k,\ell)$-sum-free subsets have cardinality $N/(k + \ell) + o(N)$, for all $(k, \ell)$, see~\cite{Eberhard15,JW20}), we split our proof into two cases: when $|A\tri (a\cdot A)|$ is small (close to having multiplicative structures), and when $|A\tri (a\cdot A)|$ is large (far away from having multiplicative structures). We mainly consider the case $a=3$ here since the Fourier coefficients appearing in the later proofs contain a multiplicative character mod $3$. The first case is resolved by an application of the Littlewood--Paley theorem, and the proof we given also works for sum-free sets. In the second case, since the main factors in the Fourier coefficients are not multiplicative, we carefully sieve out small prime factors, and apply a variant of the weak Littlewood conjecture. The nontrivial lower bound of $\om(N)$ eventually comes from the largeness of $|A\tri (3\cdot A)|$. For convenience, we make the following definition.

\begin{definition}\label{def: geometric}
We say a set $A\subseteq \ZZ^{>0}$ is a {\bf $(n,c)$-geometric set}, if $|A\tri (n\cdot A)|\ll |A|^c$ for an absolute constant $c<1$.
\end{definition}

The paper is organized as follows. In the next section, we deal with the case when $|A\tri (3\!\cdot\! A)|$ is small (we actually prove a more general result there). In Section 3, we prove a generalized version of the weak Littlewood conjecture, which is used later in the proof. In Section 4, we prove the case when $|A\tri (3\cdot A)|$ is large, and finish the proof of Theorem~\ref{thm:main}.

\subsection*{Notation} Given a set $A$ and a positive integer $k$, we use $kA$ to denote the set $\{a_1+\dots+a_k:a_i\in A \text{ for }1\leq i\leq k\},$ and use $k\cdot A$ to denote the set $\{ka:a\in A\}$. For every $\theta\in\RR/\ZZ$, we write $e(\theta)=e^{2\pi i \theta}$.
We use the standard Vinogradov notation. That is, $f\ll g$ means $f=O(g)$, and $f\asymp g$ if $f\ll g$ and $f\gg g$.


\section{When $A$ is geometric}

In this section, we study the size of the largest sum-free sets when the host set $A$ is structured. Let us first recall that in \cite[Proposition 1.4]{Bourgain97}, Bourgain proved that 
\begin{equation}\label{eq: bourgain}
\mathscr{M}_{(2,1)}(N)\geq\frac{N}{3}+c(\log N)^{-1}\left\|\sum_{m\in A}\cos(2\pi m \theta)\right\|_{L^1(\RR/\ZZ)}.
\end{equation}
Hence if the $L^1$-norm of $\sum_{m\in A}\cos(2\pi m \theta)$ is large, and we will establish such a lower bound when $A$ is geometric (see Definition \ref{def: geometric}). We remark that if $A$ is not geometric in general the $L^1$-norm of $\sum_{m\in A}\cos(2\pi m \theta)$ can be $\ll \log N$. 

\begin{definition}
    Let $\cp_{3,n}$ be the collection of intervals $[n+3^k,n+3^{k+1})\cap\NN$, where $k\geq0$ is an integer and $n\in\ZN$ is a shift. Let $A$ be a set of $N$ positive integers. We say that $A$ is {\bf $(3,c)$-lacunary}, if there is an $n\in\ZN$ and a subset $\cp\subset\cp_{3,n}$ with $|\cp|\gg N^c$, such that each interval in $\cp$ contains at least one element of $A$, and the intervals in $\cp$ form a cover of $A$.
\end{definition}
 We have the following observation.

\begin{lemma}\label{lem: geometric implies lacunary}
    Let $c>0$ and let $A$ be a set of $N$ positive integers. If $A$ is $(3,c)$-geometric, then $A$ is $(3,1-c)$-lacunary. 
\end{lemma} 
\begin{proof}
    Let $\cp_{3,0}$ be the collection of intervals $I_k:=[3^k,3^{k+1})\cap\NN$ for integers $k\geq 0$. Let $I$ be the set of indices that $I_k\cap A\neq\varnothing$ for $k\in I$. Let us also partition positive integers into collection of $3$-chains
    \[
    \ZZ_+=\bigcup_{x\in\ZZ_+,3\nmid x}\mathcal{C}(x), \quad\text{where }\mathcal{C}(x)=\bigcup_{i=0}^\infty \{3^ix\}. 
    \]
  Observe that $A$ being $(3,c)$-geometric implies that $A$ intersects $\ll N^c$ many $3$-chains nontrivially, simply as for every $x$ with $3\nmid x$, 
  \[
  |(\mathcal{C}(x)\cap A)\tri (3\cdot (\mathcal{C}(x)\cap A))|\geq 2. 
  \]
  By pigeonhole principle, there is at least one $x$ with $3\nmid x$ such that
  \[
  |\mathcal{C}(x)\cap A|\gg N^{1-c}. 
  \]
 Since different elements in $\mathcal{C}(x)$ lies in different $I_k$'s in $\cp_{3,0}$, we conclude that $A$ is $(3,1-c)$-lacunary. 
\end{proof}

The main purpose of the section is to prove the following proposition.

\begin{proposition}[Large sum-free subsets in lacunary sets]\label{prop: main lemma section 2}
   Let $c>0$, and let $A$ be a set of $N$ positive integers. If $A$ is $(3,c)$-lacunary then
   \[
   \M_{(2,1)}(A)-\frac{N}{3}\gg_c N^{\frac{c}{4}}. 
   \]
\end{proposition}

We will in fact prove a stronger form in view of the analytic sum-free conjecture 
(Conjecture~\ref{conj:main2}). See Proposition \ref{prop: main lemma section 2, second form}.

Heuristically when the set $A$ is $(3,c)$-lacunary, in some sense the distribution of $A$ is not far away from  a union of long geometric progressions, and we expect that approximately there is a square root cancellation for $\|\sum_{m\in A}e(mx)\|_{L^1(\RR/\ZZ)}$. To make this observation rigorous, we use the Littlewood--Paley theorem.
\begin{theorem}[Littlewood--Paley]
\label{L-P-theorem}
Let $g(x)$ be the trigonometric series
\begin{equation*}
    g(x)=\sum_{n=1}^\infty a_n e(nx).
\end{equation*}
For the sequence $\{a_n\}$, we consider the following auxiliary truncated function $\De_k$ defined as
\begin{equation*}
    \De_k(x)=\sum_{n=n_{k-1}+1}^{n_k} a_ne(nx),
\end{equation*}
where $n_0=0$, $n_1=1$, $n_{k+1}/n_k\geq\al>1$. Then for any $1<p<\infty$.
\begin{equation*}
    \Bigg\|\Big(\sum_{k=1}^\infty|\De_k|^2\Big)^{1/2}\Bigg\|_{L^p(\RR/\ZZ)}\leq C_{p,\al}\|g\|_{L^p(\RR/\ZZ)}.
\end{equation*}
\end{theorem}
The proof of the Littlewood--Paley theorem can be found in \cite[Chapter XV, Theorem 4.11]{Zygmund}. 
 The next lemma gives us a key estimate for lacunary sets.
\begin{lemma}\label{lem: lacunary L1}
Assume that $\{a_n\}_{n=1}^N$ is $(3,c)$-lacunary. Define $g(x)$ as 
\begin{equation*}
    g(x)=\sum_{n=1}^Ne(a_nx).
\end{equation*}
Then $\|g\|_{L^1(\RR/\ZZ)}\gg N^{c/3}$.
\end{lemma}

\begin{proof}
As $\{a_n\}_{n=1}^N$ is $(3,c)$-lacunary, by translating the set if necessary it can be assumed that there is a collection of $\gg N^c$ intervals of the shape $[3^k, 3^{k+1})$ that cover $A$ and that each contains at least one element from $A$.

Let $\phi_{k}$ be the indicator function of the interval $[3^k,3^{k+1})$. We denote by $\De_k(g)$ the Fourier truncation
\begin{equation*}
    \De_k(g)(x)=\sum_{n=1}^N\phi_{k}(a_n)e(a_nx).
\end{equation*}
By the Littlewood--Paley theorem (Theorem  \ref{L-P-theorem}), we have for any $1<p<\infty$,
\begin{equation}
\label{L-P}
    \|g\|_{L^p(\RR/\ZZ)}\geq C_p\Bigg\|\Big(\sum_{k=0}^\infty|\De_kg(x)|^2\Big)^{1/2}\Bigg\|_{L^p(\RR/\ZZ)}=C_p\Bigg\|\Big(\sum_{k\in E}|\De_kg(x)|^2\Big)^{1/2}\Bigg\|_{L^p(\RR/\ZZ)}.
\end{equation}
Here the set $E$ contains all the positive integers $k$ satisfying $[3^k,3^{k+1})\in\cp$, so $|E|=|\cp|\gg N^c$. We bound the right hand side of equation \eqref{L-P} using H\"older's inequality so that
\begin{equation*}
    \Bigg\|\Big(\sum_{k\in E}|\De_kg(x)|^2\Big)^{1/2}\Bigg\|_{L^p(\RR/\ZZ)}\geq |E|^{-1/2}\Bigg\|\Big(\sum_{k\in E}|\De_kg(x)|\Big)\Bigg\|_{L^p(\RR/\ZZ)},
\end{equation*}
which clearly implies
\begin{equation*}
    \Bigg\|\Big(\sum_{k\in E}|\De_kg(x)|^2\Big)^{1/2}\Bigg\|_{L^p(\RR/\ZZ)}\geq |E|^{-1/2}\sum_{k\in E}\bigg\|\De_kg(x)\bigg\|_{L^1(\RR/\ZZ)}.
\end{equation*}
Using H\"older's inequality again, we get $\|\De_kg\|_2^2\leq\|\De_kg\|_1\|\De_kg\|_\infty$, and this implies $\|\De_kg\|_{L^1}\geq1$ uniformly in $k$. Therefore,
\begin{equation}
\label{lower-bound-Lp}
    \Bigg\|\Big(\sum_{k\in E}|\De_kg(x)|^2\Big)^{1/2}\Bigg\|_{L^p(\RR/\ZZ)}\geq |E|^{1/2}\gg N^{c/2}.
\end{equation}
Since $\|g\|_p^p\leq\|g\|_1\|g\|_\infty^{p-1}$, we can bound $\|g\|_{L^p}^p$ easily by
\begin{equation}
\label{lower-bound-L1}
    \|g\|_{L^1(\RR/\ZZ)}\geq N^{1-p}\|g\|_{L^p(\RR/\ZZ)}^p
\end{equation}
Finally, we combine estimates \eqref{L-P}, \eqref{lower-bound-Lp} and \eqref{lower-bound-L1} to finish the proof of this lemma, by choosing $p=1+c/6$.
\end{proof}

\vspace{3mm}

Now we are going to prove the following stronger form of Proposition \ref{prop: main lemma section 2}, using the same argument in the proof of \cite[Proposition 1.4]{Bourgain97} together with Lemma~\ref{lem: lacunary L1}. As mentioned in the introduction, the following proposition would imply Proposition~\ref{prop: main lemma section 2}. 

\begin{proposition}\label{prop: main lemma section 2, second form}
Let $c>0$, and let $A$ be a set of $N$ positive integers. If $A$ is $(3,c)$-lacunary then
     there exists a maximal sum-free set $\Omega\subseteq\RR/\ZZ$ that
\[
\max_{x\in\RR/\ZZ}\sum_{n\in A}\Big(\mathbbm{1}_{\Omega}-\frac{1}{3}\Big)(nx)\gg_c N^{\frac{c}{4}}.
\]
\end{proposition}

\begin{proof}
Let $\Omega=(1/3,2/3)\subseteq \RR/\ZZ$, and it is easy to check that $\Omega$ is sum-free in $\RR/\ZZ$. Define $\mathbbm{1}_\Omega$ as the characteristic function of $\Omega$, and let $f=\mathbbm{1}_\Omega-1/3$ be the balanced function of $\mathbbm{1}_\Omega$. By orthogonality of characters we have
\[
\widehat{f}(n)=\begin{cases}
0&\quad\text{if }n=0,\\
\widehat{\mathbbm{1}_\Omega}(n)&\text{otherwise.}
\end{cases}
\]
When $n>0$,
\begin{align*}
    \widehat{f}(n)&=\int_{\RR/\ZZ} \mathbbm{1}_\Omega(x)e(-nx)\d x\\
    &=\frac{1}{2\pi i n}\Big(-e\Big(-\frac{2n}{3}\Big)+e\Big(-\frac{n}{3}\Big)\Big)\\
    &=\frac{1}{\pi n}e\Big(\frac{n}{2}\Big)\sin\Big(\frac{n\pi}{3}\Big). 
\end{align*}
Therefore we obtain
\begin{align}
f(x)&=\sum_{n\neq0}\widehat{f}(n)e(nx)=\sum_{n\neq0}\frac{1}{\pi n}e\Big(\frac{n}{2}\Big)\sin\Big(\frac{n\pi}{3}\Big)e(nx)\nonumber \\
&=-\frac{\sqrt3}{\pi}\sum_{n\geq1}\frac{\chi(n)}{n}\cos(2\pi nx), \label{eq: f}
\end{align}
where $\chi(n)$ is a nontrivial multiplicative character mod $3$, that is
\[
\chi(n)=\begin{cases}
1\quad&\text{when } n\equiv1\pmod3,\\
-1&\text{when } n\equiv2\pmod3,\\
0&\text{otherwise.}
\end{cases}
\]

Define
\[
F(x)=\sum_{m\in A}f(mx). 
\]
Since $f$ is a balanced function, we have $\int_{\RR/\ZZ}F=0$, and this implies that 
\begin{equation}\label{eq: L1 - max}
\max_{x\in\RR/\ZZ} F(x)\geq \frac{1}{2}\big\|F\big\|_{L^1(\RR/\ZZ)}. 
\end{equation}
Let $P\asymp N^2$ be a prime, and let $\MM$ be the collection of square-free integers generated by primes smaller than $P$. Let $\mu$ be the M\"obius function, so by equation (\ref{eq: f}),
\begin{align*}
    \sum_{k\in\MM}\frac{\mu(k)\chi(k)}{k}\sum_{m\in A} f(mkx)&=-\frac{\sqrt3}{\pi}\sum_{m\in A,\,n\geq1}\frac{\chi(n)}{n}\cos(2\pi mnx)\sum_{k\in\MM,\, k\mid n}\mu(k)\\
    &=-\frac{\sqrt3}{\pi}\sum_{\substack{m\in A\\ n\geq1, n\in\mathcal{N}}}\frac{\chi(n)}{n}\cos(2\pi mnx),
\end{align*}
where $\N$ is the set of integers $n$ such that for every $p<P$, $\gcd(n,p)=1$. 

Therefore, by Minkowski's inequality we have
\begin{align*}
    \Bigg\Vert \sum_{k\in\MM}\frac{\mu(k)\chi(k)}{k}\sum_{m\in A} f(mkx) \Bigg\Vert_{L^1(\RR/\ZZ)}&\gg \Bigg\Vert \sum_{m\in A}\cos(2\pi mx)\Bigg\Vert_{L^1(\RR/\ZZ)}\\
    &\quad - \Bigg\Vert \sum_{\substack{m\in A\\ n>1, n\in\mathcal{N}}}\frac{\chi(n)}{n}\cos(2\pi mnx) \Bigg\Vert_{L^1(\RR/\ZZ)}.
\end{align*}
Via the Cauchy--Schwarz inequality and Plancherel, the second term is bounded by
\begin{equation*}
    \Bigg\Vert \sum_{\substack{m\in A\\ n>1, n\in\mathcal{N}}}\frac{\chi(n)}{n}\cos(2\pi mnx) \Bigg\Vert_{L^1(\RR/\ZZ)}\leq C|A| P^{-1/2}.
\end{equation*}
Note that $P\asymp N^2$. By Mertens' estimate, we have
\begin{align*}
   \Bigg\Vert \sum_{m\in A}\cos(2\pi mx)\Bigg\Vert_{L^1(\RR/\ZZ)} &\ll \sum_{k\in\MM} \frac{|\mu(k)|}{k}\big\Vert F(x)\big\Vert_{L^1(\RR/\ZZ)} + O(1).\\
   &\ll \prod_{p<P}\Big(1+\frac{1}{p}\Big)\big\Vert F(x)\big\Vert_{L^1(\RR/\ZZ)}\asymp \log N\big\Vert F(x)\big\Vert_{L^1(\RR/\ZZ)}.
\end{align*}
Since $A$ is $(3,c)$-lacunary, we invoke Lemma~\ref{lem: lacunary L1} to get
\begin{equation}
\nonumber
    \Bigg\Vert \sum_{m\in A}\cos(2\pi mx)\Bigg\Vert_{L^1(\RR/\ZZ)}\gg N^{c/3},
\end{equation}
which implies
\[
\big\Vert F(x)\big\Vert_{L^1(\RR/\ZZ)}\gg \frac{N^{c/3}}{\log N}.
\]
Finally, we use estimate (\ref{eq: L1 - max}) to conclude $\max_{x\in\RR/\ZZ}F(x)\gg N^{c/4}$. 
\end{proof}

\section{A density estimate}

In this section, we prove a generalization of the McGehee--Pigno--Smith theorem \cite{MPS}, based on the ideas given by Bourgain~\cite{Bourgain97}.
Recall that the weak Littlewood problem~\cite{Littlewood} is to ask to estimate 
\[
I(N):=\min_{A\subseteq\mathbb{Z}, |A|=N}\int_{\RR/\ZZ}\Big|\sum_{n\in A}e(nx)\Big|\d x.
\]
The conjecture, $I(N)\gg\log N$, is resolved by McGehee, Pigno, and Smith \cite{MPS}, and independently by Konyagin \cite{K81}. 

Let $\cn_1$ be the set of natural numbers that does not contains 1, and only contains prime factors at least $Q$, where $Q\asymp(\log N)^{100}$ is a prime. We will use the following lemma from \cite[Section~5]{Bourgain97}.

\begin{lemma}
\label{bourgain-lemma}
Let $A$ be a finite subset of $\ZZ^+$ with $|A|=N$. For all $R\geq1$, we define
\begin{equation*}
    A_R=\{m\in A:m<R\}.
\end{equation*}
Also, we use ${\rm Proj}_R\sum a_ke(kx)$ to denote the truncated sum $\sum_{|k|\leq R}a_ke(kx)$. Assume $|a_n|\leq 1$ and $Q>(\log N)^{20}$. Then there is an absolute big constant $C$, such that
\begin{equation*}
    \Bigg\|{\rm Proj}_R\sum_{\substack{n\in\cn_1, m\in A}}\frac{a_n}{n}e(mnx)\Bigg\|_2< CP^{-1/15}|A_R|^{1/2}.
\end{equation*}
\end{lemma}

Now we are able to prove our technical lemma. The proof basically follows the  arguments used in \cite{JW20} (which is also based on the ideas in \cite{MPS} and \cite{Bourgain97}), but with more explanations and a stronger conclusion. 

\begin{lemma}
\label{technical-lemma}
Let $B=\{m_1,\ldots,m_M\}$ be a finite subset of $\ZN^{>0}$ and let $Q>(\log M)^{100}$. Assume that $w:\ZN^{>0}\to\ZC$ is a weight and $|a_n|=O(1)$. Then there exists a function $\Phi(x)$ with $\|\Phi\|_\infty<10$ such that 
\begin{equation}
\label{technical-estimate-1}
    \bigg|\bigg\langle\sum_{j=1}^Me(m_jx)w(m_j),\Phi(x)\bigg\rangle\bigg|\gg\sum_{j=1}^M\frac{|w(m_j)|}{j};
\end{equation}
while for any $\be\in \ZZ$,
\begin{equation}
\label{technical-estimate-2}
    \bigg|\bigg\langle\sum_{\substack{n\in\cn_1, m\in B}}\frac{a_n}{n} e(\beta mnx), \Phi(x)\bigg\rangle\bigg|\leq C(\log M)^{-2}.
\end{equation}
Here $c,C$ are two absolute constants.
\end{lemma}
\begin{proof}
Let $k_0$ be the largest natural number that $10^{6k_0}<M$. We group $B$ into disjoint subsets $\{B_k\}_{k=0}^{k_0}$ such that for $0\leq k\leq k_0-1$, $|B_k|=10^{6k}$. Here $B_0=\{m_1\}$, $B_1=\{m_2,\ldots,m_{10^6+1}\},\cdots,$ and $B_{k_0}=A\setminus(\bigcup_{k\leq k_0-1}B_k)$. From the construction we know $|B_{k_0}|\asymp 10^{6k_0}$. Let $\tau:\ZN^{>0}\to\ZS^1$ be the argument function that $\tau(m)w(m)\geq0$. For each $B_k$, we define 
\begin{equation*}
    \wt P_{k}=\frac{1}{|B_k|}\sum_{m\in B_k} e(mx)\tau(m).
\end{equation*}
Let $I_k=[a_k,b_k]$ be the interval with $a_k=\min\{m:m\in B_k\}$, $b_k=\max\{m:m\in B_k\}$, and let $\xi_k$ be the center of $I_k$. We also define
\begin{equation*}
    P_k=\wt P_{k}\ast \left(e(\xi_k x)F_{|I_k|}\right),
\end{equation*}
where $F_C=\sum_{|m|\leq C}\frac{C-|m|}{C}e(mx)$ is the $C$-F\'ejer kernel. Consequently,
\begin{equation}
\label{support-P-k}
    {\rm supp}(\wh P_k)={\rm supp}(\wh{\wt P_k})\subset I_k,
\end{equation}
and for any $m\in B_k$
\begin{equation*}
    w(m)\wh{P_k}(m) >10^{-6k-1}|w(m)|.
\end{equation*}

This shows that the functions $P_k$ are good test functions. However, the function $\sum_k P_k(x)$ has one drawback: It is not distributed evenly on the torus. That is, the $L^\infty$-norm $\sum_k P_k(x)$ is comparably large. 

To overcome this difficulty, for each $P_k$, we construct a function $Q_k$ serving as a ``compensator". Specifically, let $\ch$ be the Hilbert transform in $L^2(\RR/\ZZ)$ that $\wh{\ch{f}}(n)=-i{\rm sgn}(n)\wh f(n)$, so that when $f$ is a real-valued function, $\ch f$ is also real-valued. We define
\begin{equation}
\label{Q-k}
    Q_k=\Big(e^{-(|\ti P_{k}|-i\ch[|\ti P_{k}|])}\Big)\ast F_{|I_k|}.
\end{equation}
Since the Fourier series of $e^{-(|\ti P_{k}|-i\ch[|\ti P_{k}|])}$ is supported in non-positive integers, 
\begin{equation}
\label{support-Q-k}
    {\rm supp}\big(\wh Q_k\big)\subset[-|I_k|,0].
\end{equation}
Using the inequality that $|e^{-z}-1|\leq |z|$ if $z\in\ZC$ and ${\rm Re}(z)\geq0$, we can easily prove
\begin{equation}
\label{l2-Q-k}
    \|1-Q_k\|_2\leq\|\ti P_k\|_2+\big\|\ch[|\ti P_{k}|]\big\|_2<2|B_k|^{-1/2}.
\end{equation}
Thus, $Q_k$ is approximately the identical function. In fact, $|Q_k|$ is relatively small when $|P_k|$ is relatively large, so $Q_k$ can help us ``mollify" the function $P_k$.

\vspace{3mm}
We will use the functions $P_k,Q_k$ to construct our test function $\Phi$. In specific, we set $\Phi_0=P_0$ and set
\begin{equation}
\label{inductive-def}
    \Phi_k=Q_{k}\Phi_{k-1}+P_k, \hspace{1cm}1\leq k\leq k_0.
\end{equation}
Define $\Phi=\Phi_{k_0}$, which has the explicit formula 
\begin{equation}
\label{explicit-def}
    \Phi=P_{k_0}+P_{k_0-1}Q_{k_0}+P_{k_0-2}Q_{k_0-1}Q_{k_0}+\cdots+P_0Q_1\cdots Q_{k_0}.
\end{equation}

We claim $\|\Phi\|_\infty<10$. To see this, we first recall the basic inequality: $\frac{a}{10}+e^{-a}\leq1$ if $a\geq0$. Then, observing $|P_0|=1$ and
\begin{equation*}
    \Big\|\frac{1}{10}|P_k|+|Q_k|\Big\|_\infty\leq\Big\|\big(\frac{1}{10}|\wt P_k|+e^{-|\wt P_k|}\big)\ast F_{|I_k|}\Big\|_\infty\leq\Big\|\frac{1}{10}|\wt P_k|+e^{-|\wt P_k|}\Big\|_\infty\leq1,
\end{equation*}
we argue inductively using \eqref{inductive-def} to conclude our claim.

\vspace{3mm}
Next, we will verify \eqref{technical-estimate-1}. We will prove that for any $m\in B_k$, 
\begin{equation}
\label{pointwise-esti}
    \big|\wh\Phi(m)-\wh{P}_k(m)\big|\leq 10^{-1}\big|\wh{P_k}(m)\big|=\frac{1}{10|B_k|}.
\end{equation}
In fact, using the support condition \eqref{support-Q-k} and the equation \eqref{explicit-def}, we have
\begin{equation*}
    \wh\Phi(m)-\wh{P}_k(m)=\wh{P}_{k_0}(m)+\wh{P_{k_0-1}}\ast\wh{Q_{k_0}}(m)+\cdots+\wh{P_k}\ast(1-(Q_{k_0}\cdots Q_{k+1})^\wedge)(m),
\end{equation*}
which, combining the support condition of $\wh{P_k}$ in \eqref{support-P-k}, equals to
\begin{equation*}
    \sum_{j=k}^{k_0-1}\wh{P_j}\ast(1-(Q_{k_0}\cdots Q_{j+1})^\wedge)(m).
\end{equation*}
We estimate the above quantity using the equality
\begin{equation}
\nonumber
    1-Q_{k+1}\cdots Q_{k_0}=(1-Q_{k+1})+Q_{k+1}(1-Q_{k+2})+\cdots+(1-Q_{k_0})Q_{k+1}\ldots Q_{k_0-1},
\end{equation}
so that
\begin{equation*}
    |\wh\Phi(m)-\wh{P}_k(m)|=\bigg|\sum_{j=k}^{k_0-1}\wh{P_j}\ast(1-(Q_{k_0}\cdots Q_{j+1})^\wedge)(m)\bigg|\leq\sum_{j=k}^{k_0-1}\|P_j\|_2\sum_{l=j}^{k_0-1}\|1-Q_l\|_2.
\end{equation*}
Since $\|P_j\|_2\leq |B_j|^{-1/2}$ and since \eqref{l2-Q-k}, the right hand side of the above inequality can be bounded as
\begin{equation*}
    \sum_{j=k}^{k_0-1}\|P_j\|_2\sum_{l=j}^{k_0-1}\|1-Q_{l+1}\|_2\leq2\sum_{j=k}^{k_0-1}10^{-3j}\sum_{l=j}^{k_0-1}10^{-3(l+1)},
\end{equation*}
which implies what we need that
\begin{equation*}
    |\wh\Phi(m)-\wh{P}_k(m)|\leq10^{-3k-2}\leq 10^{-1}|\wh{P_k}(m)|.
\end{equation*}

\vspace{3mm}

As a consequence of \eqref{pointwise-esti}, for any $m\in B_k$, 
\begin{equation*}
    {\rm Re} (w\wh{\Phi})(m)>\frac{1}{2}w(m)\wh{P}_k(m)\geq 10^{-6k-1}|w(m)|.
\end{equation*}
We use the above inequality to sum up all $m\in B$ to get
\begin{equation*}
    \bigg|\bigg\langle\sum_{j=1}^Me(m_jx)w(m_j),\Phi(x)\bigg\rangle\bigg|\geq\sum_{j=1}^M{\rm Re}(w\wh\Phi)(m_j)\gg\sum_{j=1}^M\frac{|w(m_j)|}{j},
\end{equation*}
and this gives \eqref{technical-estimate-1}.

\vspace{3mm}

Finally, we remark that the proof of \eqref{technical-estimate-2} is given in \cite{JW20} Section 3, with the help of Lemma \ref{bourgain-lemma}. At this point, we complete the proof of the lemma.
\end{proof}

As an application of Lemma \ref{technical-lemma}, we have the following corollary:
\begin{corollary}\label{cor: key}
Let $B=\{m_1,\ldots,m_M\}$ be a finite subset of $\ZN^{>0}$ and let $Q>(\log M)^{100}$. Recall that $\cn_1$ is the set of natural numbers that does not contain 1 and only contains prime factors at least $Q$. Assume $|a_n|= O(1)$. Then for any $\Gamma\subset\ZZ$ with $|\Gamma|\leq \log M$, we have
\begin{equation*}
    \Bigg\|\sum_{j=1}^Me(m_jx)w(m_j)+\sum_{\substack{n\in\cn_1, m\in B}}\Big(\sum_{\beta\in\Gamma}\frac{a_n}{n}e(\beta mnx)\Big)\Bigg\|_1\geq c\sum_{j=1}^M\frac{|w(m_j)|}{j}-o(1).
\end{equation*}
\end{corollary}
\begin{proof} We apply Lemma \ref{technical-lemma} to obtain a function $\Phi(x)$ satisfying \eqref{technical-estimate-1} and \eqref{technical-estimate-2}. Then 
\begin{eqnarray}
\nonumber
    &&\Bigg\|\sum_{j=1}^Me(m_jx)w(m_j)+\sum_{\substack{n\in\cn_1, m\in B}}\Big(\sum_{\beta\in\Gamma}\frac{a_n}{n}e(\beta mnx)\Big)\Bigg\|_1\|\Phi\|_\infty\\ \nonumber
    &&\geq \Big|\Big\langle\sum_{j=1}^Me(m_jx)w(m_j),\Phi(x)\Big\rangle\Big|-\sum_{\beta\in\Gamma}\Big|\Big\langle\sum_{\substack{n\in\cn_1, m\in B}}\frac{a_n}{n} e(\beta mnx), \Phi(x)\Big\rangle\Big|\\ \nonumber
    &&> c\sum_{j=1}^M\frac{|w(m_j)|}{j}-o(1),
\end{eqnarray}
as desired.
\end{proof}


\section{When $A$ is not geometric}

In this section, we consider the case when the host set $A$ is geometrically distributed, in the sense that $|A\tri 3\cdot A|\gg N^c$ for some positive constant $c>0$. We will focus on finding the largest $(2,4)$-sum-free in $A$. Let $\Omega_1=(1/6,1/3)\subseteq \RR/\ZZ$, and let $\Omega_2=(2/3,5/6)\subseteq \RR/\ZZ$. It is clear that both $\Omega_1$ and $\Omega_2$ are $(2,4)$-sum-free in $\RR/\ZZ$. Let $\mathbbm{1}_{\Omega_t}$ be the indicator function of $\Omega_t$ for $t=1,2$. Given $A\subseteq \NN^{>0}$ of size $N$, let $\M_{(2,4)}(A)$ be the size of the maximum $(2,4)$-sum-free subset of $A$. Again we have
\begin{align*}
    \M_{(2,4)}(A)\geq\max_{x\in\RR/\ZZ}\sum_{n\in A}\mathbbm{1}_{\Omega_t}(nx), 
\end{align*}
for $t=1,2$. We introduce the balanced function $f_t:\RR/\ZZ\to\CC$ defined by $f_t=\mathbbm{1}_{\Omega_t}-\frac{1}{6}$. Hence,
\[
 \widehat{f_t}(n)=
 \begin{cases}
 0\quad & \text{ if } n=0,\\
 \widehat{\mathbbm{1}_{\Omega_t}}(n) & \text{ otherwise}.
 \end{cases}
\]
When $n>0$, the Fourier coefficient $\wh{f}_t(n)$ is
\begin{align*}
    \widehat{f_t}(n)&=\int_{\TT}\mathbbm{1}_{\Omega_t}(x)e(-nx)d\mu(x)\\
    &=\frac{1}{2\pi in}\Big(e\big(-\frac{(t-1)n}{2}-\frac{n}{3}\big)-e\big(-\frac{(t-1)n}{2}-\frac{n}{6}\big)\Big)\\
   &=\frac{1}{\pi n}e\Big(-\frac{(2t-1)n}{4}\Big)\sin\Big(\frac{n\pi}{6}\Big).
\end{align*}
Hence, for $t=1,2$ we have
\begin{align*}
    &f_t(x)=\sum_{n\neq0}\widehat{f_t}(n)e(nx)=\sum_{n\neq0}\frac{1}{\pi n}e\Big(-\frac{(2t-1)n}{4}\Big)\sin\Big(\frac{n\pi}{6}\Big)e(nx).
\end{align*}
Denote by 
\begin{equation}
\label{gt}
    g_t(x)=\sum_{n\in A}f_t(nx).
\end{equation}
We will prove that either $\|g_1\|_1\gg \log N/\log\log N$ or $\|g_2\|_1\gg \log N/\log\log N$. However, it seems hard to estimate $\|g_t\|_1$ directly. In order to get around this difficult, we consider their sum $f_1+f_2$ and difference $f_1-f_2$. Let $\Gamma(x):=f_1(x)+f_2(x)$ be the sum so that
\begin{equation}
\label{Gamma}
    \Gamma(x)=\frac{2}{\pi}\sum_{n\geq1}\frac{(-1)^{n}}{n}\sin\Big(\frac{n\pi}{3}\Big)\cos(4\pi nx).
\end{equation}
Also, we let $\Lambda(x)=f_1(x)-f_2(x)$ be the difference and let 
\[
\gamma(n)=\begin{cases}
1&\text{ when }n\equiv1\pmod4,\\
-1&\text{ when }n\equiv3\pmod4,\\
0&\text{ otherwise,}
\end{cases}
\]
so that we can express $\La(x)$ as
\begin{align}\label{eq: Gx}
    \Lambda(x)&=\frac{4}{\pi}\sum_{n\geq1}\frac{\gamma(n)}{n}\sin\Big(\frac{n\pi}{6}\Big)\sin(2\pi nx)\nonumber\\
    &=\frac{4}{\pi}\Bigg(\sum_{\substack{n\geq1\\ 3\nmid n, 2\nmid n}}\frac{1}{2n}\sin(2\pi nx)-\sum_{\substack{n\geq1\\ 3\mid n,2\nmid n}}\frac{1}{n}\sin(2\pi nx)\Bigg).
\end{align}

We first deal with the function $\Gamma(x)$. Recall that $\cn_1$ is the set of positive integers $m$ such that $m$ only contains prime factors larger than $Q\asymp(\log N)^{100}$. We also define $\cn_2$ be the set of square-free integers generated by primes that are at most $Q$. Since $(-1)^n\sin(n\pi/3)=-\sqrt3\chi(n)/2$ where $\chi(n)$ is a multiplicative character mod $3$, we can sieve out the small prime factors in \eqref{Gamma} by
\begin{equation}
\label{Gamma-sieved}
    \sum_{t\in \cn_2}\frac{\mu(t)\chi(t)}{t}\sum_{m\in A}\Gamma(mtx)=-\frac{\sqrt3}{\pi}\sum_{m\in A}\Bigg(\cos(4\pi mx)+\sum_{n\in \cn_1}\frac{\chi(n)}{n}\cos(4\pi nmx)\Bigg),
\end{equation}
where $\mu$ is the M\"obius function.

Next, we consider $\Lambda(x)$. Since the coefficients $\gamma(n)\sin(n\pi/6)$ are not multiplicative, $\La(x)$ is more difficult to handle. As shown in equation \eqref{eq: Gx}, $\La(x)$ can be partitioned into two parts according to the divisibility by the number 3. This motivates us to first sieve out those integers $n$ that $3 \!\mid\! n$, by a restricted M\"obius function defined only on integers divisible by $3$. In this way, except for the first term, all other terms with significant contribution in the second part cancel out, while the first part remains unchanged. Then, we use another sieve for the first part in a similar fashion. It turns out that we can combine these two steps to one by using the M\"obius function directly as our sieve. In fact,
\begin{align*}
  \sum_{m\in\cn_2}\frac{\mu(m)}{m}\Lambda(mx)
  =\frac{4}{\pi}\sum_{n\geq1}\frac{1}{n}\sin(2\pi nx) \sum_{\substack{m\in\cn_2, m\mid n}}\mu(m)\gamma\Big(\frac{n}{m}\Big)\sin\Big(\frac{n\pi}{6m}\Big).
\end{align*}
Depending on the divisibility of $n$ by 3 and $9$, the term $\sum_{\substack{m\in\cn_2, m\mid n}}\mu(m)\gamma(\frac{n}{m})\sin(\frac{n\pi}{6m})$ has the expression

\begin{align*}
    I_1^o(n)=\sum_{m\in\cn_2, m\mid n}\frac12\mu(m), \hspace{.5cm}\text{if }3\nmid n, 2\nmid n,\quad I_1^e(n)=\sum_{m\in\cn_2, m\mid n,2\mid m}\frac12\mu(m), \hspace{.5cm}\text{if }3\nmid n, 2\mid n\\
    I_2^o(n)=\sum_{m\in\cn_2,m\mid n,3\nmid m}-\mu(m)+\sum_{m\in\cn_2,m\mid n,3\mid m}\frac{1}{2}\mu(m),\hspace{.5cm}\text{if }3\mid n, 2\nmid n,\text{ but }9\nmid n\qquad\\
    I_2^e(n)=\sum_{m\in\cn_2,m\mid n,3\nmid m,2\mid m}-\mu(m)+\sum_{m\in\cn_2,m\mid n,6\mid m}\frac{1}{2}\mu(m),\hspace{.5cm}\text{if }6\mid n,\text{ but }9\nmid n\qquad\\
   I_3^o(n)=\sum_{m\in\cn_2, m\mid n}-\mu(m),\hspace{.5cm}\text{if }9\mid n,2\nmid n,\quad I_3^e(n)=\sum_{m\in\cn_2, m\mid n, 2\mid m}-\mu(m),\hspace{.5cm}\text{if }9\mid n, 2\mid n 
\end{align*}
By the inclusive-exclusive principle, for $n\notin \cn_1\cup 2\cdot \cn_1\cup 3\cdot \cn_1\cup 6\cdot \cn_1$, $I_1^{o,e}(n)$ is always $0$ unless $n=1,2$, and $I_3^{o,e}(n)$ is always $0$. For $I_2^{o,e}(n)$, note that
\[
\sum_{\substack{m\in\cn_2,m\mid n\\ 3\mid m}}\frac{1}{2}\mu(m)=\sum_{\substack{m\in\cn_2,m\mid n\\ 3\nmid m}}-\frac{1}{2}\mu(m),
\]
which implies that $I_2^{o,e}(n)$ is $0$ unless $n=3,6$. Therefore, we get
\begin{align}
&\Lambda_1(x):=\sum_{t\in\cn_2}\sum_{m\in A}\frac{\mu(t)}{t}\Lambda(tmx)\label{eq:Gx1}\\
=&\, \frac{4}{\pi}\sum_{m\in A}\Bigg(\frac{1}{2}\sin(2\pi mx)-\frac{1}{4}\sin(4\pi mx)-\frac{1}{2}\sin(6\pi mx)+\frac{1}{4}\sin(12\pi mx) \nonumber\\
&\,+\sum_{n\in\cn_1\cup2\cdot\cn_1\cup3\cdot\cn_1\cup6\cdot\cn_1}\frac{\eta(n)}{n}\sin(2\pi nmx)\Bigg), \nonumber
\end{align}
where $\eta$ is defined as
\[
\eta(n)=
\begin{cases}
\frac{1}{2}&\text{ when }n\in\cn_1,\\
-\frac{1}{2}&\text{ when }n\in2\cdot \cn_1,\\
-\frac{3}{2}&\text{ when }n\in3\cdot \cn_1,\\
\frac{3}{2}&\text{ when }n\in6\cdot \cn_1.
\end{cases}
\]
Note that $\Lambda_1$ indeed has the expression
\begin{align*}
    \Lambda_1(x)=&\, \frac{2}{\pi}\sum_{m\in A}\Bigg(\sin(2\pi mx)-\frac{1}{2}\sin(4\pi mx)-\sin(6\pi mx)+\frac{1}{2}\sin(12\pi mx)\\
    &\,+\sum_{n\in\N_1}\frac{1}{n}\Big(\sin(2\pi nmx)-\frac{1}{2}\sin(4\pi nmx)-\sin(6\pi nmx)+\frac{1}{2}\sin(12\pi nmx)\Big)\Bigg).
\end{align*}
Let $B=A\tri 3\cdot A$, so by our assumption on the ambient set $A$, $|B|\gg N^c$. For any number $m\in A\tri (3\cdot A)$, we define
\begin{equation*}
    \eps(m)=\begin{cases}
    1&\text{ when }m\in A\setminus (3\cdot A),\\
    -1&\text{ when }m\in (3\cdot A) \setminus A,\\
    0&\text{ otherwise.}
\end{cases}
\end{equation*}
and
\begin{equation*}
    \eps'(m)=\begin{cases}
    -\frac{1}{2}&\text{ when }m\in (2\cdot A)\setminus (6\cdot A),\\
    \frac{1}{2}&\text{ when }m\in (6\cdot A) \setminus (2\cdot A),\\
     0&\text{ otherwise.}
\end{cases}
\end{equation*}
We can thus simplify $\Lambda_1(x)$ as
\begin{align}\label{eq: G1}
    \Lambda_1(x)=&\,\frac{2}{\pi}\sum_{m\in B}\Big(\eps(m)\sin(2\pi mx)+ \sum_{n\in\cn_1}\frac{\eps(m)}{n}\sin(2\pi nmx)\Big) \\
    &\, +\frac{2}{\pi}\sum_{m\in 2\cdot B}\Big(\eps'(m)\sin(2\pi mx)+ \sum_{n\in\cn_1}\frac{\eps'(m)}{n}\sin(2\pi nmx)\Big).\nonumber
\end{align}
Another important observation is that $(\eps+\eps')$ is supported on $B\cup (2\cdot B)$ and $|(\eps+\eps')(m)|\geq 1/2$ on its support. 

Finally, we combine \eqref{Gamma-sieved} and \eqref{eq: G1} to get
\begin{align}
\nonumber
    &-\frac{2\sqrt3 \pi}{3}\sum_{t\in \cn_2}\frac{\mu(t)\chi(t)}{t}\sum_{m\in A}\Gamma(mtx)+\frac{2\sqrt3 \pi}{3}\sum_{t\in \cn_2}\frac{\mu(t)\chi(t)}{t}\sum_{m\in A}\Gamma(3mtx)\\ \nonumber
    &+\frac{\sqrt3 \pi}{3}\sum_{t\in \cn_2}\frac{\mu(t)\chi(t)}{t}\sum_{m\in A}\Gamma(2mtx)-\frac{\sqrt3 \pi}{3}\sum_{t\in \cn_2}\frac{\mu(t)\chi(t)}{t}\sum_{m\in A}\Gamma(6mtx)\\ \nonumber
    &+\frac{i\pi}{2}\sum_{t\in\cn_2}\frac{\mu(t)}{t}\sum_{m\in A}\Lambda(2tmx)\\ \nonumber
    =\ &\sum_{m\in B\cup(2\cdot B)}\Bigg((\epsilon+\eps')(m)\cos(4\pi mx)+\sum_{n\in\cn_1}\frac{\chi(n)(\eps+\eps')(m)}{n}\cos(4\pi nmx)\Bigg)\\ \nonumber
    &+i\sum_{m\in B\cup(2\cdot B)}\Bigg((\epsilon+\eps')(m)\sin(4\pi mx)+\sum_{n\in\cn_1}\frac{(\eps+\eps')(m)}{n}\sin(4\pi nmx)\Bigg)\\
     \nonumber
    =\ &\sum_{m\in B\cup(2\cdot B)}e^{4\pi imx}(\eps+\eps')(m)\\ \label{final}
    &+\sum_{m\in B\cup(2\cdot B), n\in\cn_1}\frac{(\eps+\eps')(m)}{n}\Big((\chi(n)+1)e^{4\pi nmx}+ (\chi(n)-1)e^{-4\pi nmx}\Big).
\end{align}
Denote also
\begin{equation}
\label{la-a-ga-a}
    \La_A(x)=\sum_{m\in A}\La(mx), \hspace{.5cm}\Ga_A(x)=\sum_{m\in A}\Ga(mx).
\end{equation}
Now we can employ Corollary~\ref{cor: key} and the triangle inequality to \eqref{final} to obtain 
\begin{align*}
  4\sum_{t\in\cn_2}\frac{|\mu(t)|}{t}\big\Vert \Gamma_A\big\Vert_{L^1(\RR/\ZZ)}+ \sum_{t\in\cn_2}\frac{|\mu(t)|}{t}\big\Vert \Lambda_A\big\Vert_{L^1(\RR/\ZZ)}\gg\log N.
\end{align*}
Mertens' estimate tells us
\[
\sum_{t\in\cn_2}\frac{1}{t}\ll \prod_{p<Q}\Big(1+\frac{1}{p}\Big)\asymp \log\log N. 
\]
Hence we have
\[
\max\big\{ \big\Vert \Gamma_A\big\Vert_{L^1(\RR/\ZZ)}, \big\Vert \Lambda_A\big\Vert_{L^1(\RR/\ZZ)}\big\}\gg\frac{\log N}{\log\log N}.
\]
This implies that there is $t\in\{1,2\}$, such that $\Vert g_t\Vert\gg\log N/\log\log N$, and since $g_t$ is balanced, we get $\max_{x\in\RR/\ZZ}g_t(x)\gg\log N/\log\log N$. \medskip

With all tools in hand we are going to prove our main theorem. 
\begin{proof}[Proof of Theorem~\ref{thm:main}]
Fix $k\geq1$. We first assume $|A\tri 3\cdot A|\gg N^{1/2}$. By the result proved earlier in this section, we may assume that for $\Omega_1=(1/6,1/3)$, there is $x_0\in \RR/\ZZ$ such that 
\[
\sum_{n\in A}\mathbbm{1}_{\Omega_1}(nx_0)\geq\frac{N}{6}+c\frac{\log N}{\log\log N},
\]
where $c>0$ is an absolute constant. Consider the continuous group homomorphism $\chi:\RR/\ZZ\to\RR/\ZZ$ with $\chi(x)=kx$ for every $x$. Then the Bohr set $\chi^{-1}(\Omega_1)$ is a union of $k$ disjoint open intervals $I_1,\dots,I_k$ in $\RR/\ZZ$, each of which has measure $1/6k$. It is also easy to see that $I_t$ is $(2k,4k)$-sum-free for every $1\leq t\leq k$. Indeed, suppose that $I_1$ is not $(2k,4k)$-sum-free, then there are $6k$ elements $a_1,\dots,a_{2k},b_1,\dots,b_{4k}$ in $I_1$ such that $\sum_{j=1}^{2k} a_j=\sum_{j=1}^{4k} b_j$. We may assume $a_1\leq \dots \leq a_{2k}$ and $b_1\leq \dots \leq b_{4k}$. Define $\alpha_r=\tfrac{1}{k}\sum_{j=rk+1}^{(r+1)k}a_j$, and $\beta_s=\tfrac{1}{k}\sum_{j=sk+1}^{(s+1)k}b_j$ for all $r\in\{0,1\}$ and $s\in\{0,1,2,3\}$. Since $I_1$ is an interval, $\alpha_0,\alpha_1$ and $\beta_0,\dots,\beta_3$ all belong to $I_1$, and $\sum_{j=0}^1\alpha_j=\sum_{j=0}^3\beta_j$. Now, using the fact that $\chi$ is a group homomorphism, we have $\sum_{j=0}^1\chi(\alpha_j)=\sum_{j=0}^3\chi(\beta_j)$, and this contradicts the fact that $\Omega_1$ is $(2,4)$-sum-free. 

Pick $x_1$ such that $x_0=kx_1$. By pigeonhole principle and the fact $\mathbbm{1}_{\Om_t}(kx)=\mathbbm{1}_{I_1}(x)+\cdots+\mathbbm{1}_{I_k}(x)$, there is $t_0\in\{1,\dots,k\}$ such that
\[
\sum_{n\in A}\mathbbm{1}_{I_t}(nx_1)\geq\frac{N}{6k}+\frac{c}{k}\frac{\log N}{\log\log N},
\]
this finishes the proof of the first case.

Now let us assume $|A\tri 3\cdot A|\ll N^{1/2}$. By Lemma \ref{lem: geometric implies lacunary}, 
 $A$ is at least a $(3,1/2)$-lacunary set. Let $\Omega=(1/3,2/3).$ By Proposition~\ref{prop: main lemma section 2, second form}, there is $y_0\in\RR/\ZZ$, such that 
\[
\sum_{n\in A}\mathbbm{1}_{\Omega}(ny_0)\geq\frac{N}{3}+cN^{1/8},
\]
for some constant $c>0$. Let $y_1=y_0/2k$, then similarly there is an open interval $I\subseteq \RR/\ZZ$, such that $I$ has length $1/6k$, $I$ is $(2k,4k)$-sum-free, and
\[
\sum_{n\in A}\mathbbm{1}_{I}(ny_1)\geq\frac{N}{6k}+\frac{c}{2k}N^{1/8},
\]
this finishes the proof. 
\end{proof}

\section*{Acknowledgements}
The authors thank George Shakan for carefully reading the first draft of the paper and for making many useful comments. They would also like to thank the anonymous referee for a detailed report with many helpful suggestions.


\bibliographystyle{amsplain}
\bibliography{ref}

\end{document}